\documentclass[12pt,oneside,english]{amsart}
\usepackage[T1]{fontenc}
\usepackage[utf8]{inputenc}
\usepackage{geometry}
\usepackage{fancyhdr}
\usepackage{color}
\usepackage{amstext}
\usepackage{amsthm}
\usepackage{amssymb}
\usepackage{mathdots}
\usepackage{esint}

\makeatletter
\numberwithin{equation}{section}
\numberwithin{figure}{section}
\theoremstyle{plain}
\newtheorem{thm}{\protect\theoremname}
  \theoremstyle{plain}
  \newtheorem{lem}[thm]{\protect\lemmaname}
  \theoremstyle{remark}
  \newtheorem*{rem*}{\protect\remarkname}
  \theoremstyle{remark}
  \newtheorem{rem}[thm]{\protect\remarkname}
  \theoremstyle{plain}
  \newtheorem{cor}[thm]{\protect\corollaryname}

\usepackage[english]{babel}
\usepackage{fancyhdr}
\providecommand{\lemmaname}{Lemma}

\providecommand{\theoremname}{Theorem}
\providecommand{\corollaryname}{Corollary}
\newcommand{\abs}[1]{\ensuremath{|#1|}}

\newcommand{\Abs}[1]{\ensuremath{\left|#1\right|}}

\newcommand{\norm}[2]{\ensuremath{|\!|#1|\!|_{#2}}}

\newcommand{\Norm}[2]{\ensuremath{\left|\!\left|#1\right|\!\right|_{#2}}}

\usepackage{babel}
\providecommand{\lemmaname}{Lemma}
  
  \providecommand{\remarkname}{Remark}
\providecommand{\theoremname}{Theorem}

\usepackage{babel}
\providecommand{\lemmaname}{Lemma}
  
\providecommand{\theoremname}{Theorem}

\usepackage{babel}
\providecommand{\theoremname}{Theorem}

\usepackage{babel}
\providecommand{\theoremname}{Theorem}

\usepackage{mathtools}

\DeclarePairedDelimiterX{\inp}[2]{\langle}{\rangle}{#1, #2}

\makeatother

\usepackage{babel}
  \providecommand{\corollaryname}{Corollary}
  \providecommand{\lemmaname}{Lemma}
  \providecommand{\remarkname}{Remark}
\providecommand{\theoremname}{Theorem}

\begin{document}

\title{Interpolation without commutants}

\author{Oleg Szehr}

\address{Dalle Molle Institute for Artificial Intelligence (IDSIA) - SUPSI/USI, Manno, Switzerland.}

\email{oleg.szehr@posteo.de}

\author{Rachid Zarouf}

\address{Aix-Marseille Universit\'e, EA-4671 ADEF, ENS de Lyon, Campus Universitaire de Saint-J\'er\^ome, 40 Avenue Escadrille Normandie Niemen, 13013 Marseille, France.}

\email{rachid.zarouf@univ-amu.fr}

\address{Department of Mathematics and Mechanics, Saint Petersburg State University,
28, Universitetski pr., St. Petersburg, 198504, Russia.}

\email{rzarouf@gmail.com}



\begin{abstract}
We introduce a \lq\lq{}dual-space approach\rq\rq{} to mixed Nevanlinna-Pick/Cara\-th\'eo\-do\-ry-Schur interpolation in Banach spaces $X$ of holomorphic functions on the disk. Our approach can be viewed as complementary to the well-known commutant lifting approach of D.~Sarason and B.~Nagy-C.Foia\c{s}. We compute the norm of the minimal interpolant in $X$ by a version of the Hahn-Banach theorem, which we use to extend functionals defined on a subspace of kernels without increasing their norm. This Functional extensions lemma plays a similar role as Sarason\rq s Commutant lifting theorem but it only involves the predual of $X$ and no Hilbert space structure is needed. As an example, we present the respective Pick-type interpolation theorems for Beurling-Sobolev spaces.
\end{abstract}

\maketitle

\section{\label{subsec:the_duality_method}Introduction}

\subsection{The commutant lifting approach to interpolation theory}

Given a finite sequence of distinct points $\lambda={\left(\lambda_{i}\right)}_{i=1}^{n}$
in $\mathbb{D}$ and another finite sequence $w={\left(w_{i}\right)}_{i=1}^{n}$
in $\mathbb{C}$, the Nevanlinna-Pick interpolation problem is to
find necessary and sufficient conditions for the existence of $f\in\mathcal{H}ol(\mathbb{D})$
that is bounded by $1$ and that interpolates the data, i.e.~$\Norm{f}{H^{\infty}}:=\sup_{z\in\mathbb{D}}\abs{f(z)}\leq1$
and $f(\lambda_{i})=w_{i}$. The classical solution of G. Pick \cite{PG}
and (later) R. Nevanlinna \cite{NR1,NR2} asserts that such $f$ exists
if and only if the \lq\lq{}Pick-matrix\rq\rq{}
\[
\left(\frac{1-w_{i}{\overline{w_{j}}}}{1-\lambda_{i}{\overline{\lambda_{j}}}}\right)_{1\leq i,\,j\leq n}
\]
is positive-semidefinite. The celebrated commutant-lifting approach
of D.~Sarason \cite{SD} and B. Nagy-C.~Foia\c{s} \cite{NF2,AM} established
an operator-theoretic perspective on the interpolation problem. The
main point is to view $H^{\infty}$ as a multiplier algebra of the
Hardy space $H^{2}$ (of holomorphic functions, whose Taylor coefficients
are square-summable) and to identify $f\in H^{\infty}$ with the respective
multiplication operator $Mult_{f}:H^{2}\rightarrow H^{2}$, $\phi\mapsto f\phi$.
In other words the norm of $f\in H^{\infty}$ equals to the operator
norm of $Mult_{f}$. The Nevanlinna-Pick problem asks for conditions
on the restriction $M_{f}$ of $Mult_{f}$ to a subspace corresponding
to $f(\lambda_{i})=w_{i}$ such that $M_{f}$ can be extended to the
whole of $H^{2}$ maintaining $\Norm{Mult_{f}}{}\leq1$. More precisely,
for a Blaschke product
\[
B=\prod_{\lambda_{i}\in\lambda}b_{\lambda_{i}},\qquad b_{\lambda_{i}}=\frac{z-\lambda_{i}}{1-\overline{\lambda_{i}}z}
\]
the range $BH^{2}$ of the multiplication operator $Mult_{B}$ is
a linear subspace of $H^{2}$ of functions that vanish at $\lambda$.
The orthogonal complement
\begin{align*}
K_{B}:=H^{2}\ominus BH^{2},
\end{align*}
consists of rational functions whose poles are $1/\overline{\lambda_{i}}$, $i=1\dots n$. Let $M_{f}$ be the compression of the multiplication operator $Mult_{f}$
from $H^{2}$ to $K_{B}$, i.e.
\begin{align*}
M_{f}:\:K_{B} & \rightarrow K_{B}\\
\phi & \mapsto(P_{B}\circ Mult_{f})({\phi})=P_{B}(f\phi),
\end{align*}
where $P_{B}$ is the orthogonal projection from $H^{2}$ to $K_{B}$.
It is clear that $\Norm{M_{f}}{}\leq\Norm{f}{H^{\infty}}$ and
that $M_{f}$ only depends on the values $\{f(\lambda_{i})=w_{i}\}_{i=1}^{n}$.
Sarason\rq s main result \cite{SD,AM} asserts that if $M_{g}$ is any
operator that commutes with $M_{f}$ on $K_{B}$ then $M_{g}$ can
be extended to an operator $Mult_{g}$ that commutes with $Mult_{f}$
on $H^{2}$ \emph{without increasing the operator's norm}. In conclusion
one finds
\begin{align*}
\Norm{M_{f}}{} & =\inf\left\{ \Norm{Mult_{g}}{}:P_{B}\circ Mult_{g}=M_{f}\right\} \\
 & =\inf\left\{ \Norm{g}{H^{\infty}}:\:g\in H^{\infty},\:g(\lambda_{i})=f(\lambda_{i})=w_{i},\:\forall i\right\}
\end{align*}
where the first equality is a consequence of Sarason\rq s result and
the second holds by {the} construction of $K_{B}$. It is elementary to
check that $\Norm{M_{f}}{}\leq1$ is equivalent to the Pick-matrix
being positive-semidefinite~\cite{NN2}.

This new perspective allowed for various generalizations of the classical
Nevanlinna-Pick problem. To mention a few, it has been noticed that
the assumption of non-degeneracy of the $\lambda_{i}$ is not principal.
The purely degenerate case $\lambda_{1}=...=\lambda_{n}=0$ corresponds
to the prescription of the first $n$ Taylor-coefficients of $f$.
The respective interpolation theory has been studied long before as
the Carath\'eodory-Schur interpolation problem. Thus the commutant lifting
approach provided a unified framework with a simultaneous discussion
of two classical problems of interpolation theory. Second, Sarason\rq s
result was generalized by B. Nagy-C. Foia\c{s} \cite{NF2,AM}, whose commutant
lifting theorem asserts that any operator commuting with $A$ can
be lifted to an operator commuting with any unitary dilation of $A$
without increasing its norm. Sarason\rq s lemma is the special case that
$A$ is the multiplication operator $Mult_{z}$ on $H^{2}$. A more
general result is the intertwining lifting theorem of Nagy-Foia\c{s} \cite{FF,NF1}.
As a consequence the $H^{2}$-specific discussion has been generalized
to \emph{Reproducing Kernel Hilbert Space }(RKHS) to study interpolation
problems of the respective commutant algebras: Let a positive-definite function 
{$(z,\zeta)\longmapsto\kappa(z,\zeta)$} on $\mathbb{D}\times\mathbb{D}$
be given
{(i.e. $\sum_{i,j}{a}_{i}\overline{a}_{j}\kappa(\lambda_{i},\lambda_{j})>0$
for all finite subsets ${\left(\lambda_{i}\right)}\subset\mathbb{D}$ and all
non-zero families of complex numbers $\{a_{i}\}$)} and let
{$\kappa_{\zeta}=\kappa(\cdot,\zeta)$.}
Following Aronszajn~\cite{AN} there exists a unique Hilbert space
of functions $\mathcal{H}(\kappa)$, such that $\kappa$ enjoys the
\emph{reproducing kernel property}, i.e.~for all $f\in\mathcal{H}(\kappa)$
it holds
\[
f(\zeta)=\left\langle f,\,\kappa_{\zeta}\right\rangle _{\mathcal{H}(\kappa)}.
\]
{When $\kappa$ is holomorphic in the first variable and antiholomorphic
in the second this yields a RKHS $\mathcal{H}(\kappa)$ of holomorphic
functions on $\mathbb{D}$.}
 The algebra of multipliers $\mathbb{M}_{\kappa}$
is a Banach algebra of functions $\phi$ for which $f\phi\in\mathcal{H}(\kappa)$
for each $\phi\in\mathcal{H}(\kappa)$ and the norm is the norm of
the corresponding multiplication operator on $\mathcal{H}(\kappa)$.
For the Cauchy kernel $\kappa_{\zeta}(z)=\frac{1}{1-\bar{\zeta}z}$
we obtain $\mathcal{H}(\kappa)=H^{2}$ and $\mathbb{M}_{\kappa}=H^{\infty}$.
It is a natural question to ask for which kernels $\kappa$ apart from the Cauchy kernel a ``Nevanlinna-Pick theorem'' holds. Assuming that $\kappa$ is a so-called {complete} Nevanlinna-Pick kernel, i.e. it satisfies the identity 
\[
\kappa({z},{\zeta})-\frac{\kappa({z},\mu)\kappa(\mu,{\zeta})}{\kappa(\mu,\mu)}=F_{\mu}({z},{\zeta})\kappa({z},{\zeta})
\]
for some $\mu\in\mathbb{D}$, $\kappa(\mu,\mu)\neq0$ and some positive semidefinite function $F_{\mu}$ on $\mathbb{D}\times\mathbb{D}$ such that $\abs{F_{\mu}({z},{\zeta})}<1$, it is shown in \cite{QP} that
the Nevanlinna-Pick theorem holds~{mutatis mutandis}: There exists a multiplier $f$ of norm at most $1$ which satisfies the interpolation condition $f(\lambda_{i})=w_{i}$ if and only if
\[
\left[\kappa(\lambda_{i},\,\lambda_{j})(1-{{w_{i}}\overline{w_{j}}})\right]_{1\leq i,\,j\leq n}\geq0.
\]
The interesting article~\cite{SH} contains a detailed discussion of complete Nevanlinna-Pick kernels in the context of Dirichlet spaces. A general commutant lifting theorem for spaces with Nevanlinna-Pick kernels is proved in \cite{BTV}.
\subsection{Our approach and its motivation}

In this article we study interpolation problems beyond the context
of RKHS. Let $X$ be a Banach space that is continuously embedded
into $Hol(\mathbb{D})$. Our goal is to obtain information on the
interpolation quantity
\[
I_{X}(\lambda,\,w)=\inf\left\{ \Norm{f}{X}:\:f\in X,\:f^{(j)}(\lambda_{i})=w_{i}^{(j)}\:,\,1\leq i \leq n,\,0\leq j<n_{i}\right\} ,
\]
where $\lambda_{i}$ carries degeneracy $n_{i}$, $$w=\left(w_{1}^{(0)},w_{1}^{(1)},\dots,w_{1}^{(n_{1}-1)},\dots,w_{s}^{(0)},w_{s}^{(1)},\dots,w_{s}^{(n_{s}-1)},\dots\right)$$
and $f^{(j)}$ stands for the $j$-th derivative of $f$. This definition
covers mixed problems of Nevanlinna-Pick and Carath\'eodory-Schur type
\cite{NN2}. Our approach is closely related to the established commutant lifting
theory. A major common point will lie in the role of the space $K_{B}$
and the compressions $M_{f}$ of the multiplication operator $Mult_{f}:X\rightarrow X$.
Our main conceptual insight might be seen in the observation that
no RKHS structure is needed to identify the space $K_{B}$ and that
the formulation of the solution in terms of the multiplication operator
(and with it the occurrence of the multiplier algebra) can be done
in an independent step. Thus we conceptually split the Hilbert space
specific commutant lifting theorem, into a \lq\lq{}Functional extension
lemma\rq\rq{} and a formulation of the solution in terms of the
multiplication operator. Just as the commutant lifting theorem allows
one to extend an operator from $K_{B}$ to $H^{2}$ without increasing
the operator norm this lemma allow us to extend functionals from $K_{B}$
to a Banach space of holomorphic functions without increasing the
norm of the functional. Our motivation is twofold,\\
 \emph{1)} on the theoretical side: The commutant lifting approach
has generated significant impact on interpolation theory, operator
theory, functional analysis and beyond. We see our method as complementary
to this approach, but for certain types of spaces it is simpler. Instead of studying
the commutant algebra of a RKHS, we will work with duals
of a class of Banach spaces. In many cases the dual space turns out
to be more tangible than the commutant algebra.\\
 \emph{2)} on the practical side: Our result is interesting from a
practical standpoint because it dramatically simplifies the numerical
computation of the quantity $I_{X}(\lambda,w)$. The search domain for the minimization is an infinite-dimensional
Banach space, and therefore the search does not admit implementation
on finite-memory and finite-precision computers. In contrast the representation
afforded by the functional extension lemma reduces the original (infinite-dimensional)
minimization problem to a search in $K_{B}$ for an optimal $n$-dimensional
vector of coefficients $\alpha=(\alpha_{1},\alpha_{2},...,\alpha_{n})\in\mathbb{C}^{n}$,
which can be obtained by applying any standard function minimization
algorithm. Consequently, we provide the theoretical basis for the
reduction of the general optimization problem to one that can actually
be implemented in practice.

\subsection{Outline of the paper}

In Section~\ref{Section1} we explain how our approach applies to a large class of Banach spaces and we compare our methods in more detail with those of D.~Sarason. Section~\ref{sec:General-result} contains our main results. Lemma \ref{thm:main}
provides an expression for $I_{X}(\lambda,\,w)$ in terms of the norm
of a functional on the space $K_{B}$. Assuming that $X$ is a unital Banach algebra Theorem~\ref{thm:shift_main_algebra} relates this expression to the compressed
multiplication operator. This theorem extends Sarason\rq s original result
to any unital algebra $X$ whose predual contains $K_{B}=K_{\lambda}$.
Corollary \ref{thm:main_3} applies interpolation theory to derive sharp
estimates on norms of functions of algebraic operators admitting an
$X$-functional calculus. We conclude Section \ref{sec:General-result} with
an application of Theorem \ref{thm:shift_main_algebra} to the case $X=W$, the \textit{Wiener algebra} of
absolutely convergent Taylor series. Section \ref{Section1} contains the details regarding $W$, its definition and our motivation to study it. Section~\ref{sec:Applications_of_1st_Th} shows applications of Lemma \ref{thm:main} to the so-called \emph{Beurling-Sobolev spaces} $X=l_{A}^{q}(\beta)$ and $X=H^{\infty}$. Section \ref{sec:Interpolation-on-a} discusses the situation in which the data $\lambda$ has an accumulation point inside the disk. In this case there is a unique solution $f$ to the interpolation problem in $X$. Finally Section 6 contains the proof of Corollary \ref{thm:NP} where we recover Pick's usual criterion.
\section{\label{Section1}Interpolation in Banach spaces}

Our approach can be seen as a \lq\lq{}duality based\rq\rq{}
discussion of interpolation theory in the Banach space $X$. We endow $Hol(\mathbb{D})$
with a formal Cauchy-type scalar product
\[
\left\langle f,\,g\right\rangle =\sum_{k\geq0}\hat{f}(k)\overline{\hat{g}(k)},
\]
where $f=\sum_{k\geq0}\hat{f}(k)z^{k}$ denotes the Taylor expansion
of $f$. For $f,g\in H^{2}$ this coincides with the usual $H^{2}$
scalar product $\left( \cdot,\,\cdot\right)_{H^{2}}$ inherited from $L^{2}$, i.e.~$\left\langle f,\,g\right\rangle=(f,\,g)_{H^{2}}$. We assume that $X$ is a dual space $X=Y'$ w.r.t.~$\left\langle \cdot,\,\cdot\right\rangle$. $Y$
will always denote the \emph{exact} predual of $X$ which means that the
norm of $g$ in $Y$ can be computed by the classical Hahn-Banach
formula
\begin{align*}
\Norm{g}{Y}:=\sup\left\{ \frac{\abs{{\langle g,\,f\rangle}}}{\Norm{f}{X}}\::\:f\in X\right\}.
\end{align*}
We will also assume that the predual $Y$ contains
the set of all analytic polynomials as a dense subset. Any function $f\in X$ can be interpreted as a functional on $Y$,
$f\mapsto\tilde{f}:=\left\langle \cdot,\,f\right\rangle $,
and for the norm we have by H\"older's inequality
\[
\Norm{f}{X}=\norm{\tilde{f}}{Y'}:= \sup\left\{ \frac{\abs{{\langle g,\,f\rangle}}}{\Norm{g}{Y}}\::\:g\in Y\right\} .
\]
The main point on which our interpolation theory is footed is that $Y$ contains the space $K_{\lambda}$ of rational functions whose poles are located at $1/\overline{\lambda_{i}}$, with possible multiplicity $n_{i}$:
\[
K_{\lambda}=:{\rm span}\{k_{\lambda_{i},\,j}:\:1\leq i\leq n,\,0\leq j<n_{i}\},
\]
where for $\lambda_{i}\neq0$, $k_{\lambda_{i},\,j}=\left(\frac{d}{d\overline{\lambda_{i}}}\right)^{j}k_{\lambda_{i}}$
and $k_{\lambda_{i}}=\frac{1}{1-\overline{\lambda_{i}}z}$ is the
Cauchy kernel at $\lambda_{i}$ while $k_{0,\,i}=z^{i}$. In other
words we only consider $X$ such that the scalar products $\left\langle f,\,k_{\lambda_{i}}\right\rangle$
are finite, which is a very mild assumption as long as $n$ is finite,
i.e.~the boundary behaviour of the kernels $k_{\lambda_{i},j}$ is
regular. Under this assumption we have for any $f\in X$ that
\begin{align*}
\left\langle f,\,k_{\lambda_{i}}\right\rangle =\sum_{k\geq0}\hat{f}(k)\lambda_{i}^{k}=f(\lambda_{i})
\end{align*}
and similarly
\begin{align*}
\left\langle f,\,k_{\lambda_{i},\,j}\right\rangle =f^{(j)}(\lambda_{i}).
\end{align*}

To compare to the setting of commutant lifting observe that no
matter what $X$ and $Y$ are it still holds that $K_{\lambda}=K_{B}=H^{2}\cap(BH^{2})^{\bot}$,
where $B$ is the finite Blaschke product corresponding to $\lambda$.
If $\lambda_{i}\neq\lambda_{j}$ for $i\neq j$ we have $K_{\lambda}=\textnormal{span}\{k_{\lambda_{i}}\}$,
which corresponds to the Nevanlinna-Pick problem. The purely degenerate
kernels $K_{\lambda}=\textnormal{span}\{z^{i}\}$ correspond to the
Carath\'{e}odory-Schur problem. Notice that
\begin{itemize}
\item Sarason computes $I_{H^{\infty}}(\lambda,\,w)$ by lifting an operator
commuting with $M_{z}$ on {$K_{B}$} to an operator commuting
with $Mult_{z}$ on the whole of $H^{2}$ without increasing its norm,
\item we compute $I_{X}(\lambda,\,w)$ by using a version of the Hahn-Banach
theorem extending a functional from the subspace $K_{\lambda}$ of
$Y$ to the whole space $Y$ without increasing its norm. Subsequently
we rewrite the norm of the functional on $K_{\lambda}$ in terms of
the compressed multiplication operator and thereby extend Sarason\rq s
original result.
\end{itemize}
In view of applications we are particularly interested in the case
that $X=W\subsetneq H^{\infty}$ is the \textit{Wiener algebra} of
absolutely convergent Taylor series
\begin{align*}
W:=\{f=\sum_{j\geq0}\hat{f}(j)z^{j}\in\mathcal{H}ol(\mathbb{D})\::\:\Norm{f}{W}:=\sum_{j\geq0}\abs{\hat{f}(j)}<\infty\}.
\end{align*}
{To the best of our knowledge neither the Nevanlinna-Pick nor the Carath\'{e}odory-Schur interpolation problem have been studied in this setup before.} 
By von Neumann\rq s inequality Hilbert space contractions admit an
$H^{\infty}$ functional calculus. As a consequence Sarason\rq s $H^{\infty}$-interpolation
theory has contributed significant insight to the study of such operators~\cite{NFBK}. Similarly \emph{Banach space contractions}
are related to a Wiener algebra functional calculus. Our interest
in $W$ comes from developing an {analogous} theory for contractions
on Banach space.

More generally we will discuss the \emph{Beurling-Sobolev spaces}
of functions $f\in\mathcal{H}ol(\mathbb{D})$ whose sequence of Taylor
coefficients $\{\hat{f}(j)\}_{j\geq0}$ are contained in the weighted
sequence spaces $l^{q}(\beta)$, $q\in[1,\,\infty]$, $\beta\in\mathbb{R}$:

\[
X=l_{A}^{q}(\beta):=\left\{ f=\sum_{j\geq0}\hat{f}(j)z^{j}\in\mathcal{H}ol(\mathbb{D})\::\:\norm{f}{l_{A}^{q}(\beta)}:=\left(\sum_{j\geq0}\abs{\hat{f}(j)}^{q}\omega_{j}^{q}\right)^{1/q}<\infty\right\} ,
\]
where $w_{0}=1$ and $w_{j}=j^{\beta}$ for $j\geq1$.
{Again the general interpolation problem in such spaces has not been previously investigated.}
Notice that $W=l_{A}^{1}(0)$ and that the Hilbert space $l_{A}^{2}(0)$
is just the standard Hardy space $H^{2}$. It is easily verified that
those spaces satisfy our assumptions and that the {norm on the} predual of $l_{A}^{q}(\beta)$
is {\norm{\cdot}{l_{A}^{p}(-\beta)}} where $p$ is the conjugate exponent of $q$:
$\frac{1}{p}+\frac{1}{q}=1$. 
We also show how our method applies
to $X=H^{\infty}$, whose predual is given by \cite[Chapter VII]{KP},
\[
H^{\infty}=(L^{1}/\overline{H_{0}^{1}})',
\]
and we recover Pick's classical result. (Here $L^{1}=L^{1}(\partial\mathbb{D})$
is the usual $L^{1}$ space of the unit circle and $H_{0}^{1}=zH^{1}$
is a subspace of the respective Hardy space $H^1$, see~\cite{NN2}.)

\section{\label{sec:General-result}Main results}

Let $X$, $Y$ and $K_{\lambda}$ be given. Our goal is to express
$I_{X}(\lambda,\,w)$ in terms of a quantity that can be determined
from $K_{\lambda}$ alone. We suppose  {for notational convenience} that $\lambda_{i}\neq\lambda_{j}$ for $i\neq j$. {In case that $\lambda_i$ carries degeneracy $n_i$ the below argumentation can be immediately extended by considering the kernels $k_{\lambda_{i},j}$, $0\leq j\leq n_i$. Plugging in definitions we have}
\begin{eqnarray*}
I_{X}(\lambda,\,w) & = & \inf\left\{ \Norm{f}{X}:\:f\in X,\:f(\lambda_{i})=w_{i},\:\forall i=1\dots n\right\} \\
 & = & \inf\left\{ \Norm{\tilde{f}}{Y'}:\:f\in X,\:\tilde{f}(k_{\lambda_{i}})={\overline{w_{i}}},\:\forall i=1\dots n\right\} ,
\end{eqnarray*}
where we have used that $\Norm{f}{X}=\Norm{\tilde{f}}{Y'}$ and $\tilde{f}(k_{\lambda_{i}})={{\langle k_{\lambda_{i}},f}\rangle=\overline{\langle f,k_{\lambda_{i}}}\rangle=\overline{w_{i}}}$.
The condition $\tilde{f}(k_{\lambda_{i}})={\overline{w_{i}}}$ means that the restriction
$\tilde{f}\big|_{K_{\lambda}}$ coincides with the functional
\begin{align*}
\tilde{k}:K_{\lambda} & \rightarrow\mathbb{C}\\
\sum_{i=1}^{n}\alpha_{i}k_{\lambda_{i}} & \mapsto\tilde{k}(\sum_{i=1}^{n}\alpha_{i}k_{\lambda_{i}})=\sum_{i=1}^{n}\alpha_{i}{\overline{w_{i}}},
\end{align*}
which is tantamount to
\[
I_{X}(\lambda,\,w)=\inf\left\{ \Norm{\tilde{f}}{Y'}:\:f\in X,\:\tilde{f}|_{K_{\lambda}}=\tilde{k}\right\}.
\]
Consequently
\[
I_{X}(\lambda,\,w)\geq\Norm{\tilde{k}}{\left(K_{\lambda},\,\Norm{\cdot}{Y}\right)\rightarrow\mathbb{C}},
\]
where
\[
\Norm{\tilde{k}}{\left(K_{\lambda},\,\Norm{\cdot}{Y}\right)\rightarrow\mathbb{C}}:=\sup_{g\in K_{\lambda},\:g\neq0}\frac{\abs{\tilde{k}(g)}}{\Norm{g}{Y}}.
\]
According to Hahn-Banach theorem \cite[Theorem 5.16, p. 104]{RW}
since $K_{\lambda}$ is a subspace of the normed linear space $Y$
and $\tilde{k}$ is a bounded linear functional on $K_{\lambda}$,
$\tilde{k}$ can be extended to a bounded linear functional on the
whole of $Y$ having the same norm as $\tilde{k}$. In other words
there exists $\widetilde{f^{*}}\in Y'$ such that $\widetilde{f^{*}}|_{K_{\lambda}}=\tilde{k}$
and $\Norm{\widetilde{f^{*}}}{Y'}=\Norm{\tilde{k}}{\left(K_{\lambda},\,\Norm{\cdot}{Y}\right)\rightarrow\mathbb{C}}$.
Thus the infimum in $I_{X}(\lambda,\,w)$ is achieved and
\[
I_{X}(\lambda,\,w)=\Norm{\tilde{k}}{\left(K_{\lambda},\,\Norm{\cdot}{Y}\right)\rightarrow\mathbb{C}}.
\]
We note that with this simple formula the interpolation problem is, in principle, solved. We have written $I_{X}(\lambda,\,w)$ exclusively
as a function of the interpolation data, which is encoded in $K_{\lambda}$.
What remains is to write the interpolation problem in {red}{a} familiar form, e.g.~in terms of the Pick matrix.
\begin{lem}[Functional extension lemma]
\label{thm:main}Let $X$ be a Banach space of holomorphic functions on $\mathbb{D}$ whose exact predual is $Y$. Let $\lambda=(\lambda_{i})_{i=1}^{n}$
be a sequence in $\mathbb{D}$ such that $K_{\lambda}\subset Y$. Defining the functional $\tilde{k}$ on $K_{\lambda}$ by
\[
\tilde{k}(k_{\lambda_{i},\,j})={\overline{w_{i}^{(j)}}},\:\forall i=1\dots n,\;\forall j=0\dots n_{i}-1
\]
the following equality holds
\[
I_{X}(\lambda,\,w)=\Norm{\tilde{k}}{\left(K_{\lambda},\,\Norm{\cdot}{Y}\right)\rightarrow\mathbb{C}}.
\]
\end{lem}
The result asserts that given ${C}>0$ there exists $f\in X$ such that
$f^{(j)}(\lambda_{i})=w_{i}^{(j)}$ and $\Norm{f}{X}\leq {C}$ iff
\begin{align}
\Abs{\sum_{i=1}^{{n}}\sum_{j=0}^{n_{i}}\alpha_{i,j}\overline{w_{i}^{(j)}}}\leq {C}\Norm{\sum_{i=1}^{{n}}\sum_{j=0}^{n_{i}}\alpha_{i,j}k_{\lambda_{i},\,j}}{Y}\label{AF}
\end{align}
holds for any sequence of complex numbers $(\alpha_{i,j})_{i,j}$. Notice that when no degeneracy is present in the data condition~\eqref{AF}
is stucturally reminiscent to the positivity of the Pick matrix
\[
\sum_{i=1}^{n}\sum_{j=1}^{n}\frac{\overline{\alpha_{i}}\alpha_{j}\left({C}^{2}-\overline{w_{i}}w_{j}\right)}{1-\overline{\lambda_{i}}\lambda_{j}}\geq0.
\]
The main difference is that the Pick condition
is quadratic in the $\alpha_{i}$, while~\eqref{AF}
is linear. Indeed the positivity of the Pick matrix is equivalent
to a bound on the $H^{2}$ operator norm of the compressed multiplication
operator, which explains the occurrence of quadratic terms. Our next
goal is thus to rewrite~\eqref{AF} in terms of the
compressed shift operator. To work with multiplication operators
it is clear that we must endow the spaces $X$, $Y$ with additional
structure. We shall assume that:
\begin{enumerate}
\item $X$ is a unital Banach algebra (i.e.~$1\in X$ and for
all $f_{1},f_{2}\in X$ , $X$ contains the product $f_{1}f_{2}\in X$
and $\Norm{f_{1}f_{2}}{X}\leq\Norm{f_{1}}{X}\Norm{f_{2}}{X}$).
\item The exact predual $Y$ of $X$ has a division property (i.e.~$g\in Y\implies\frac{g-g(0)}{z}\in Y$).
\end{enumerate}
To keep the notation simple we write briefly $S=Mult_{z}$
for the multiplication operator on $X$, $S$ is commonly called the shift
operator. The adjoint of $S$ with respect to our
Cauchy-type duality is $S^{\star}$ defined on $Y$. $S^{\star}$
is also known as the backward shift operator and satisfies
\[
S^{\star}f=\frac{f-f(0)}{z},\qquad f\in Y.
\]
It can be checked simply that $K_{\lambda}$ is invariant with respect
to $S^{\star}$. Moreover for any $\lambda\in\mathbb{D}$ and for
any analytic polynomial $g$ we have
\[
S^{\star}k_{\lambda}=\overline{\lambda}k_{\lambda},\qquad g(S)^{\star}k_{\lambda}=\overline{g(\lambda)}k_{\lambda}.
\]
\begin{thm}
\label{thm:shift_main_algebra} Let $X\subset\mathcal{H}ol(\mathbb{D})$
be a unital Banach algebra satisfying the division property and $\lambda=(\lambda_{i})_{i=1}^{n}$
be a finite sequence in $\mathbb{D}$ whose associated Blaschke
product is denoted by $B$. For any analytic polynomial
$g$ it holds that

\[
I_{X}(\lambda,g(\lambda))=\Norm{g(S)^{\star}\vert K_{\lambda}}{Y\rightarrow Y}
\]
where $Y$ is the exact predual of $X$. \end{thm}
\begin{rem}
\label{Remark_Piii}Regarding the case $X=H^{\infty}$: The proof of
Corollary \ref{thm:NP} below shows that
\[
\Norm{g(S)^{\star}\vert K_{\lambda}}{L^{1}/\overline{H_{0}^{1}}\rightarrow L^{1}/\overline{H_{0}^{1}}}=\Norm{g(S)^{\star}\vert K_{\lambda}}{H^{2}\rightarrow H^{2}},
\]
which is Sarason\rq s formulation of Pick's theorem \cite{PG}. This
way Lemma \ref{thm:main} applies to $X=H^{\infty}=(L^{1}/\overline{H_{0}^{1}})'$
and yields Pick's criterion.
\end{rem}
The main point of Theorem~\ref{thm:shift_main_algebra} is that it provides
a representation of the interpolation quantity in terms of the norm
of an operator. This allows an elementary comparison with Pick's classical
criterion of positive-semidefniteness but it also opens the doors
to interesting applications in matrix analysis. We apply the theorem
to obtain sharp estimates on norms of functions of matrices with given
spectrum. It is known that the interpolation quantity
can itself {be} seen as an upper estimate to the norm of a function
of a matrix. We show that the converse also holds. We begin by reviewing the known upper bound:

An operator $T$ is \textit{algebraic} if there exists an analytic
polynomial $g\neq0$ such that $g(T)=0$. We denote by $m_{T}$ its\textit{
minimal polynomial}, i.e. the unique monic polynomial annihilating
$T$ whose degree $\abs{m_{T}}$ is minimal. Given an algebraic operator
$T$ with spectrum in $\mathbb{D}$ we put $m=m_{T}=\prod_{i=1}^{\abs{m}}(z-\lambda_{i})$
with $\abs{\lambda_{i}}<1$ for $i=1\dots\abs{m}$. We assume that
there exists a unital algebra $X\subset\mathcal{H}ol(\mathbb{D})$
on which $T$ admits a functional calculus with constant ${c}>0$, that
is
\begin{equation}
\Norm{g(T)}{}\leq {c}\Norm{g}{X}\label{eq:funct_calc_on_X}
\end{equation}
for any polynomial $g$. In addition to assumptions (1) and
(2) we assume furthermore that $X$ satisfies the \textit{division
property
\[
\left[f\in X,\,\lambda\in\mathbb{D},\:{\rm and}\,f(\lambda)=0\right]\Rightarrow\left[\frac{f}{z-\lambda}\in X\right].
\]
}Following \cite{NN1} instead of considering
$g$ directly in inequality \eqref{eq:funct_calc_on_X}, we add multiples
of $m$ to this function and consider $h=g+mf$ with $f\in X$. This
leads to
\[
\norm{g\left(T\right)}{}\leq {c}I_{X}(\lambda,\,g(\lambda)).
\]
Following~\cite[Lemma III.6]{SO} we extend this inequality to
any rational function $\Psi$ whose set of poles $\{\xi_{i}\}_{i=1}^{p}$
is separated from the eigenvalues of $T$ considering the analytic
polynomial
\[
g(z)=\Psi\prod_{i=1}^{p}\left(\frac{m(\xi_{i})-m(z)}{m(\xi_{i})}\right),
\]
where all singularities are lifted and observe that $g(T)=\Psi(T$).
This gives 
\[
\norm{\Psi\left(T\right)}{}\leq {c}\Norm{g(S)^{\star}\vert K_{\lambda}}{Y\rightarrow Y}={c}\Norm{\Psi(S)^{\star}\vert K_{\lambda}}{Y\rightarrow Y}
\]
because $m(S)^{\star}\vert K_{\lambda}=0$.
\begin{cor}
\label{thm:main_3} In the setting of Theorem~\ref{thm:shift_main_algebra}, if $T$ admits a $c$ functional calculus on X then for any rational function $\Psi$ whose poles are distinct from the eigenvalues of $T$ it holds
that
\[
\norm{\Psi\left(T\right)}{}\leq {c}\Norm{\Psi(S)^{\star}\vert K_{\lambda}}{Y\rightarrow Y}.
\]
\end{cor}

Notice that the right hand side only depends on the norm on $Y$ and
the minimal polynomial of $T$. The bound is optimal since equality
is achieved for the compression of any multiplication operator to
$K_{\lambda}$. We formulate a corollary of the theorem for matrices.
This is achieved simply by identifying $\mathbb{C}^{\abs{m}}\cong K_{\lambda}$
and introducing an orthonormal basis of $K_{\lambda}$.

Let ${\mathcal{M}_{n}(\mathbb{C})}$ be the set of $n\times n$ complex matrices and ${M}\in {\mathcal{M}_{n}(\mathbb{C})}$ with minimal polynomial $m=m_{{M}}$. Given any particular norm $\left|\cdot\right|$ on $\mathbb{C}^{n}$ we consider the corresponding operator norm of ${M}$: $\Norm{{M}}{} =\Norm{{M}}{\left(\mathbb{C}^{n},\,\left|\cdot\right|\right)\rightarrow\left(\mathbb{C}^{n},\,\left|\cdot\right|\right)}$. We introduce a norm $\abs{\cdot}_{*}$ on $\mathbb{C}^{\abs{m}}\cong K_{\lambda}$ by
\[
\abs{{\vec{x}}}_{*}:=\norm{\sum_{j=1}^{\abs{m}}x_{j}e_{j}}{Y}
\]
 where
\begin{align*}
e_{1}=\frac{(1-\vert\lambda_{1}\vert^{2})^{1/2}}{1-\bar{\lambda}_{1}z},\qquad e_{j}:=\frac{(1-\vert\lambda_{k}\vert^{2})^{1/2}}{1-\bar{\lambda}_{j}z}\prod_{i=1}^{j-1}b_{\lambda_{i}} & ,\qquad j=1\dots\abs{m}
\end{align*}
is the \textit{Malmquist-Walsh family}: a particular orthonormal basis
for $K_{\lambda}$ \cite[p. 137]{NN3}.  We denote by $\Norm{\cdot}{*}$
the matrix norm induced by $\abs{\cdot}_{*}$.
\begin{cor}
\label{gen_of_Sarason_thm} Let ${M}\in{\mathcal{M}_{n}(\mathbb{C})}$ be
a complex $n\times n$ matrix, with minimal polynomial $m=\prod_{i=1}^{\abs{m}}(z-\lambda_{i}),\,\lambda_{i}\in\mathbb{D}$, and such that $M$ admits a $c$ functional calculus on X.
If $\Psi$ be any rational function whose poles are distinct from the zeroes of $m$ then
it holds
\[
\Norm{\Psi\left({M}\right)}{}\leq c\norm{\Psi(\hat{M}_{z})^\star}{*},
\]
where
\[
\left(\hat{M}_{z}\right)_{ij}=\begin{cases}
\qquad\qquad\qquad{0}\  & if\ i<j\\
\qquad\qquad\qquad\lambda_{i}\  & if\ i=j\\
(1-\abs{\lambda_{i}}^{2})^{1/2}(1-\abs{\lambda_{j}}^{2})^{1/2}\prod_{\mu=j+1}^{i-1}\left(-\bar{\lambda}_{\mu}\right)\  & if\ i>j.
\end{cases}
\]

\end{cor}

Observe that $\hat{M}_{z}$ is the matrix of $M_{z}\vert K_{\lambda}$ with respect to the \textit{Malmquist-Walsh basis} $(e_{i})_{i=1}^{\abs{m}}$. Its entries are computed in \cite[Proposition III.5]{SO}. The above corollary provides the theoretical foundation for efficient computation of sharp upper estimates to norms of rational functions of matrices. It says that for given spectrum for any norm and any rational function a sharp upper estimate is given in terms of the $*$-norm of the matrix $\Psi(\hat{M}_{z})^\star$. The quantity $\norm{\Psi(\hat{M}_{z})^\star}{*}$ can be computed with less effort as compared to the interpolation quantity $I_{X}$, which involves an optimization over an infinite set.

We recall that  every Banach space contraction admits a functional calculus on the Wiener algebra $W$ with constant ${c}=1$. Indeed given $M\in\mathcal{M}_{n}(\mathbb{C})$ such that $\norm{M}{}\leq1$ for some induced matrix norm $\Norm{\cdot}{}$ and $g(z)=\sum_{j=0}^{d}\hat{g}(j)z^{j}$ we have
\[
\norm{g(M)}{}\leq\sum_{j=0}^{d}\abs{\hat{g}(j)}\norm{M^{k}}{}\leq\norm{g}{W}.
\]
Corollary~\ref{gen_of_Sarason_thm} applied to $X=W=l_{A}^{1}(0)$, whose exact predual is equipped with the norm 
\norm{\cdot}{l_{A}^{\infty}}, yields
\[
\Norm{\Psi\left({M}\right)}{}\leq \norm{\Psi(\hat{M}_{z})^\star}{l_{A}^{\infty}\rightarrow l_{A}^{\infty}}.
\]
Based on this, the authors \cite{SZ1} have recently provided an explicit class of counterexamples to Sch\"{a}ffer's conjecture \cite{SJ} about norms of inverses.

\section{\label{sec:Applications_of_1st_Th} Applications}

We illustrate the application of Lemma \ref{thm:main} for some specific choices of $X$ including the cases $X=H^2,H^{\infty},W$ and more generally $X=l_{A}^{q}(-\beta)$.

\begin{cor}
\label{thm:InterpNP_l_p} In the setting of Lemma \ref{thm:main} there exists $f\in l_{A}^{q}(-\beta)$
such that $f(\lambda_{i})=w_{i}$ and $\Norm{f}{l_{A}^{q}(-\beta)}\leq C$
iff
\[
\Abs{\sum_{i=1}^{n}\alpha_{i}{w_{i}}}\leq {C}\left(\Abs{\sum_{i=1}^{n}\alpha_{i}}^{p}+\sum_{j\geq1}j^{\beta p}\Abs{\sum_{i=1}^{n}\alpha_{i}{\lambda_{i}^{j}}}^{p}\right)^{1/p}
\]
holds for any sequence of complex numbers $(\alpha_{i})_{i=1}^{n}$.
\end{cor}
This reduces the interpolation problem on the infinite-dimensional
Beurling-Sobolev space to an optimization task over a finite-dimensional space
of rational functions $K_{\lambda}={\rm span}\{k_{\lambda_{i}}:\:i=1\dots n\}$, which is much better accessible to computers.
When $X$ is the Hardy space $H^{2}=l_{A}^{2}(0)$ this formula can be simplified. Then it holds that
\[
\Norm{\sum_{i=1}^{n}\alpha_{i}k_{\lambda_{i}}}{H^{2}}^{2}=\sum_{i=1}^{n}\sum_{j=1}^{n}\alpha_{i}\overline{\alpha_{j}}\left\langle k_{\lambda_{i}},\,k_{\lambda_{j}}\right\rangle =\sum_{i=1}^{n}\sum_{j=1}^{n}\frac{\alpha_{i}\overline{\alpha_{j}}}{1-\overline{\lambda_{i}}\lambda_{j}}
\]
and the above condition reduces to
\[
\sum_{i=1}^{n}\sum_{j=1}^{n}\frac{\overline{\alpha_{i}}\alpha_{j}\left({C}^{2}-\overline{w_{i}}w_{j}(1-\overline{\lambda_{i}}\lambda_{j})\right)}{1-\overline{\lambda_{i}}\lambda_{j}}\geq0.
\]

If in addition $X$ is a Banach algebra the optimization task on $K_{\lambda}$ is a consequence of the representation of $I_{X}(\lambda,\,w)$ in terms of the backward shift operator $S^{\star}$ in Theorem~2.

In the case $X=H^{\infty}$ classical theorem by
G.~Pick follows from Remark~\ref{Remark_Piii}.
\begin{cor}[G. Pick \cite{PG}, R. Nevanlinna \cite{NR1,NR2}]
\label{thm:NP} Let ${\left(\lambda_{i}\right)_{i=1}^{n}}$ be a sequence
of distinct points in $\mathbb{D}$ and ${\left(w_{i}\right)_{i=1}^{n}}$ be a
sequence in $\mathbb{C}$. There exists $f\in H^{\infty}$
such that $f(\lambda_{i})=w_{i}$ and $\Norm{f}{H^{\infty}}\leq {C}$
if and only if
\[
\sum_{i=1}^{n}\sum_{j=1}^{n}\frac{\overline{\alpha_{i}}\alpha_{j}\left({C}^{2}-\overline{w_{i}}w_{j}\right)}{1-\overline{\lambda_{i}}\lambda_{j}}\geq0
\]
for any sequence of complex numbers  ${\left(\alpha_{i}\right)_{i=1}^{n}}$ .
\end{cor}

If $X$ is the Wiener algebra $W$ the following characterization
follows directly from Theorem \ref{thm:shift_main_algebra}.
\begin{cor}
\label{thm:InterpNP_W} Let ${\left(\lambda_{i}\right)_{i=1}^{n}}$ be a
sequence of distinct points in $\mathbb{D}$ and ${\left(w_{i}\right)_{i=1}^{n}}$
be a sequence in $\mathbb{C}$. There exists $f\in W$
such that $f(\lambda_{i})=w_{i}$ and $\Norm{f}{W}\leq {C}$
if and only if
\[
\sup_{j\geq0}\Abs{\sum_{i=1}^{n}\alpha_{i}{w_{i}}{\lambda_{i}^{j}}}\leq {C}\sup_{j\geq0}\Abs{\sum_{i=1}^{n}\alpha_{i}{\lambda_{i}^{j}}}
\]
for any sequence of complex numbers ${\left(\alpha_{i}\right)_{i=1}^{n}}$. 
\end{cor}

\section{\label{sec:Interpolation-on-a}Interpolation on a sequence with accumulation point}
For finite data Lemma \ref{thm:main} says that given $C>0$ there exists $f\in X$ such that $f(\lambda_{i})=w_{i}$ and $\Norm{f}{X}\leq C$ iff 
\[
\Abs{\sum_{i=1}^{n}\alpha_{i}\overline{f(\lambda_{i})}}\leq C\Norm{\sum_{i=1}^{n}\alpha_{i}k_{\lambda_{i}}}{Y}
\]
for any sequence of complex numbers ${\left(\alpha_{i}\right)_{i=1}^{n}}$.
If $\lambda=(\lambda_{i})_{i\geq1}$ is an infinite sequence of distinct points which contains an accumulation point in $\mathbb{D}$ then $f$ exists and is unique.
Furthermore $f$ satisfies the above inequality for any sequence $(z_{i})_{i}$
in $\mathbb{D}$. We state this result in the case that $X$ is a Beurling-Sobolev
space $X=l_{A}^{q}(-\beta),$ $1\leq q\leq\infty$, $\beta\in\mathbb{R}$.
The proof, which we {sketch below}, is an adaptation of \cite[Lemma 4.1]{FH} and of \cite[Theorem 4.2]{FH}.
\begin{thm}
\label{thm:accum_pt_l_p_spaces} Let $\lambda=(\lambda_{i})_{i\geq1}$
be a sequence of distinct points with an accumulation point
in $\mathbb{D}$, $w=(w_{i})_{i\geq1}$ be a sequence in $\mathbb{C}$
and ${C}>0$. Let $X=l_{A}^{q}(-\beta)$ and $p, q$
be conjugate exponents, $\frac{1}{p}+\frac{1}{q}=1$. If
\begin{equation}
\Abs{\sum_{i=1}^{n}\alpha_{i}{w_{i}}}\leq {C}\left(\Abs{\sum_{i=1}^{n}\alpha_{i}}^{p}+\sum_{j\geq1}j^{\beta p}\Abs{\sum_{i=1}^{n}\alpha_{i}{\lambda_{i}^{j}}}^{p}\right)^{1/p}\label{eq:assumpt_ineq_acc_pt}
\end{equation}
for any $n\geq1$ and any sequence of complex numbers $(\alpha_{i})_{i\geq1}$,
then there exists a unique function $f\in l_{A}^{q}(-\beta)$ such
that $f(\lambda_{i})=w_{i}$ and $\Norm{f}{l_{A}^{q}(-\beta)}\leq {C}$.
Moreover the inequality
\[
\Abs{\sum_{i=1}^{n}\alpha_{i}{f(z_{i})}}\leq {C}\left(\Abs{\sum_{i=1}^{n}\alpha_{i}}^{p}+\sum_{j\geq1}j^{\beta p}\Abs{\sum_{i=1}^{n}\alpha_{i}{z_{i}^{j}}}^{p}\right)^{1/p}
\]
holds for all sequence of distinct points $(z_{i})\subset\mathbb{D}$
and all sequence of complex numbers $(\alpha_{i})$.
\end{thm}

The existence of $f$ is a consequence of the following observations (see~\cite[Lemma 4.1]{FH} for details): Assuming first -- as in~\cite[Lemma 4.1]{FH} -- that the sequence $(\lambda_{i})_{i\geq1}$ is convergent, say to $\lambda_{0}\in\mathbb{D}$, the sequence $(w_{i})_{i}$ is bounded because $\abs{w_{i}}\leq\Norm{k_{\lambda_{i}}}{l_{A}^{p}(\beta)}$ for any $i\geq1$. It has at least one convergent subsequence
$(w_{i'})$ whose limit we shall denote by $w_{0}$. Using \eqref{eq:assumpt_ineq_acc_pt}
we write 
\[
\Abs{\sum_{i=1}^{n}\alpha_{i}\overline{w_{i}}+\alpha_{0}\overline{w_{i'}}}\leq\Norm{\sum_{i=1}^{n}\alpha_{i}k_{\lambda_{i}}+\alpha_{0}k_{\lambda_{i'}}}{l_{A}^{p}(\beta)}
\]
and observe that since $\lambda_{i'}\rightarrow \lambda_{0}$ as $i'\rightarrow\infty$
we can extend \eqref{eq:assumpt_ineq_acc_pt} by continuity to include
the index value $i=0$: 
\[
\Abs{\sum_{i=1}^{n}\alpha_{i}\overline{w_{i}}+\alpha_{0}\overline{w_{0}}}\leq\Norm{\sum_{i=1}^{n}\alpha_{i}k_{\lambda_{i}}+\alpha_{0}k_{\lambda_{0}}}{l_{A}^{p}(\beta)}.
\]
 Observe that for any $n\geq1$ we have 
\begin{align*}
\abs{w_{n}-w_{0}} & ^{p}\leq\Norm{k_{\lambda_{n}}-k_{\lambda_{0}}}{l_{A}^{p}(\beta)}^{p}\\
 & =\sum_{j\geq1}j^{\beta p}\abs{\lambda_{n}-\lambda_{0}}^{jp}\\
 & =\abs{\lambda_{n}-\lambda_{0}}^{p}\sum_{j\geq0}(j+1)^{\beta p}\abs{\lambda_{n}-\lambda_{0}}^{pj}
\end{align*}
but $\sum_{j\geq0}(j+1)^{\beta p}\abs{\lambda_{n}-\lambda_{0}}^{pj}\rightarrow_{n\rightarrow\infty}1$
and the whole sequence $(w_{n})$ actually converges to $w_{0}$.
Similarly it is shown that the sequence of divided differences
\[
\Delta(i,1)=\frac{f(\lambda_{i})-f(\lambda_{0})}{\lambda_{i}-\lambda_{0}}=\frac{w_{i}-w_{0}}{\lambda_{i}-\lambda_{0}}
\]
is also convergent. Denoting by $w_{-1}$ its limit $f$ must satisfy $f'(\lambda_{0})=w_{-1}$. Repeated application of this process leads to assertions about higher-order divided differences 
\[
\Delta(i,j+1)=\frac{\Delta(i,j)-w_{-j}}{\lambda_{i}-\lambda_{0}},\qquad j\geq1
\]
where $w_{-j}$ is the limit of the sequence $(\Delta(i,j))_{i\geq1}$. As before $(\Delta(i,j+1))_{i\geq1}$ is convergent. Denoting by $w_{-j-1}$ its limit we find that $f$ must satisfy $f^{(j+1)}(\lambda_{0})=(j+1)!w_{-j-1}$.
At the end of this process Inequality \eqref{eq:assumpt_ineq_acc_pt} can be extended by continuity so that its left-hand side becomes
\[
 \Abs{\sum_{i=1}^{n}\alpha_{i}\overline{w_{i}}+\sum_{i\geq0}\beta_{i}\frac{f^{(i)}(\lambda_{0})}{i!}}, 
 \]
 where $(\beta_{i})_{i\geq1}$ is any arbitrary sequence of complex numbers. Expanding $f$ in a Taylor series around $\lambda_{0}$ in terms of $w_{-i}$ one can show that $f$ solves the interpolation problem and express $f(z_{n})$ as a linear combination of $\left(\frac{f^{(i)}(\lambda_{0})}{i!}\right)_{i\geq0}$ to generalize \eqref{eq:assumpt_ineq_acc_pt} such that it holds for any sequence $(z_{n})$ of distinct points in a neighborhood of $\lambda_0$.

To complete the proof of Theorem~\ref{thm:accum_pt_l_p_spaces} it remains to follow the steps of \cite[Theorem~4.2]{FH}: Let $\lambda_{0}$ be an accumulation point in $\mathbb{D}$ of the sequence $(\lambda_{i})$, let $r>0$ be such that the closure of $D(\lambda_{0},r)=:\{z:\:\abs{z-\lambda_{0}}<r\}$ is contained in $\mathbb{D}$, and let $(\lambda_{i'})$ be a convergent subsequence of $(\lambda_{i})$ with $\lambda_{i'}\rightarrow\lambda_{0}$ and $\{\lambda_{i'}\}\subset D(\lambda_{0},r)$. We may use the above reasoning to conclude that there is an analytic function $f$ on $D(\lambda_{0},r)$ such that $f(\lambda_{i})=w_{i}$ for all $\lambda_{i}\in D(\lambda_{0},r)$ and which satisfies Inequality \eqref{eq:assumpt_ineq_acc_pt} for all sequence $(z_{i})$ in $D(\lambda_{0},r)$. Using the standard arguments of analytic continuation from \cite[p.~566]{FH} it is shown that $f$ can be continued to all of $\mathbb{D}$ and that Inequality \eqref{eq:assumpt_ineq_acc_pt} holds throughout $\mathbb{D}$.

\section{Proof of Corollary \ref{thm:NP}}

\begin{proof}[Proof of Corollary \ref{thm:NP}]
Step 1. We prove the lower bound
\[
I_{H^{\infty}}(\lambda,\,w)\geq \sup_{(\alpha_{i})\in\mathbb{C}^{n}}\frac{\sum_{i=1}^{n}\sum_{j=1}^{n}\frac{\overline{\alpha_{i}}\alpha_{j}\overline{w_{i}}w_{j}}{1-\overline{\lambda_{i}}\lambda_{j}}}{\sum_{i=1}^{n}\sum_{j=1}^{n}\frac{\overline{\alpha_{i}}\alpha_{j}}{1-\overline{\lambda_{i}}\lambda_{j}}}.
\]
First we observe that
\begin{eqnarray*}
\Norm{\sum_{i=1}^{n}\alpha_{i}\overline{w_{i}}k_{\lambda_{i}}}{H^{2}}^{2} & = & \sum_{i=1}^{n}\sum_{j=1}^{n}\overline{\alpha_{j}}\alpha_{i}\overline{w_{i}}w_{j}\left\langle k_{\lambda_{i}},\,k_{\lambda_{j}}\right\rangle \\
 & = & \sum_{i=1}^{n}\sum_{j=1}^{n}\frac{\alpha_{i}\overline{\alpha_{j}}\overline{w_{i}}w_{j}}{1-\overline{\lambda_{i}}\lambda_{j}}.
\end{eqnarray*}
Second because of {the} Hahn-Banach formula for the norm in $H^{2}=(H^{2})'$
we have
\begin{eqnarray*}
\Norm{\sum_{i=1}^{n}\alpha_{i}\overline{w_{i}}k_{\lambda_{i}}}{H^{2}} & = & \sup_{\Norm{h}{H^{2}}\leq1}\inp*{\sum_{i=1}^{n}\alpha_{i}\overline{w_{i}}k_{\lambda_{i}}}{h}\\
 & = & \sup_{\Norm{h}{H^{2}}\leq1}\Abs{\sum_{i=1}^{n}\alpha_{i}\overline{w_{i}}\overline{h(\lambda_{i})}}\\
 & = & \sup_{\Norm{h}{H^{2}}\leq1}\Abs{\sum_{i=1}^{n}\alpha_{i}\overline{f(\lambda_{i})}\overline{h(\lambda_{i})}}\\
 & = & \sup_{\Norm{h}{H^{2}}\leq1}\inp*{\sum_{i=1}^{n}\alpha_{i}k_{\lambda_{i}}}{fh}
\end{eqnarray*}
for any $f\in H^{\infty}$ such that $f(\lambda_{i})=w_{i}$. Applying Cauchy-Schwarz inequality we find
\begin{eqnarray*}
\Abs{\inp*{\sum_{i=1}^{n}\alpha_{i}k_{\lambda_{i}}}{fh}} & \leq & \Norm{\sum_{i=1}^{n}\alpha_{i}k_{\lambda_{i}}}{H^{2}}\Norm{fh}{H^{2}}\\
 & \leq & \Norm{\sum_{i=1}^{n}\alpha_{i}k_{\lambda_{i}}}{H^{2}}\Norm{f}{H^{\infty}}\Norm{h}{H^{2}}
\end{eqnarray*}
which yields
\[
\Norm{\sum_{i=1}^{n}\alpha_{i}\overline{w_{i}}k_{\lambda_{i}}}{H^{2}}\leq\Norm{\sum_{i=1}^{n}\alpha_{i}k_{\lambda_{i}}}{H^{2}}\Norm{f}{H^{\infty}}
\]
for any $f\in H^{\infty}$ such that $f(\lambda_{i})=w_{i}$. Therefore
\[
I_{H^{\infty}}(\lambda,\,w)\geq\frac{\Norm{\sum_{i=1}^{n}\alpha_{i}\overline{w_{i}}k_{\lambda_{i}}}{H^{2}}}{\Norm{\sum_{i=1}^{n}\alpha_{i}k_{\lambda_{i}}}{H^{2}}}
\]
for all $(\alpha_{i})\in\mathbb{C}^{n}$. It
remains to replace $\alpha_{i}$ by $\overline{\alpha_{i}}$ to complete
the proof of Step 1.

Step 2. We prove the upper bound
\[
I_{H^{\infty}}(\lambda,\,w)\leq \sup_{(\alpha_{i})\in\mathbb{C}^{n}}\frac{\sum_{i=1}^{n}\sum_{j=1}^{n}\frac{\overline{\alpha_{i}}\alpha_{j}\overline{w_{i}}w_{j}}{1-\overline{\lambda_{i}}\lambda_{j}}}{\sum_{i=1}^{n}\sum_{j=1}^{n}\frac{\overline{\alpha_{i}}\alpha_{j}}{1-\overline{\lambda_{i}}\lambda_{j}}}.
\]
To this aim we observe that according to Lemma \ref{thm:main} the
interpolation quantity $I_{H^{\infty}}(\lambda,\,w)$ is the {smallest}
$C>0$ such that
\begin{eqnarray}
\Abs{\sum_{i=1}^{n}\alpha_{i}\overline{w_{i}}} & \leq & C\Norm{\sum_{i=1}^{n}\alpha_{i}k_{\lambda_{i}}}{L^{1}/\overline{H_{0}^{1}}}\label{eq:duality_NP}\\
 & = & C\inf\left\{ \Norm{\sum_{i=1}^{n}\alpha_{i}k_{\lambda_{i}}+\overline{zg}}{L^{1}}:\:g\in H^{1}\right\} \nonumber
\end{eqnarray}
for any sequence of complex numbers $(\alpha_{i})_{i}$.
We show that
\[
C=\sup_{(\alpha_{i})\in\mathbb{C}^{n}}\frac{\sum_{i=1}^{n}\sum_{j=1}^{n}\frac{\overline{\alpha_{i}}\alpha_{j}\overline{w_{i}}w_{j}}{1-\overline{\lambda_{i}}\lambda_{j}}}{\sum_{i=1}^{n}\sum_{j=1}^{n}\frac{\overline{\alpha_{i}}\alpha_{j}}{1-\overline{\lambda_{i}}\lambda_{j}}}
\]
satisfies \eqref{eq:duality_NP}. We know from
Lemma \ref{thm:main} that there exists $f\in H^{\infty}$such that
$f(\lambda_{i})=w_{i}$ and $I_{H^{\infty}}(\lambda,\,w)=\Norm{f}{\infty}.$
We have
\begin{eqnarray*}
\sum_{i=1}^{n}\overline{\alpha_{i}}w_{i} & = & \sum_{i=1}^{n}\overline{\alpha_{i}}f(\lambda_{i})\\
 & = & \left\langle f,\,\sum_{i=1}^{n}\alpha_{i}k_{\lambda_{i}}\right\rangle =\left\langle f,\,\sum_{i=1}^{n}\alpha_{i}k_{\lambda_{i}}+\overline{zg}\right\rangle ,\qquad\forall g\in H^{1}
\end{eqnarray*}
because $\left\langle f,\,\overline{zg}\right\rangle =\left\langle zfg,\,1\right\rangle =0$.
Let 
\[
B=\prod_{i=1}^{n}\frac{z-\lambda_{i}}{1-\overline{\lambda_{i}}z}
\]
be the Blaschke product corresponding to  $\lambda=(\lambda_{i})_{i=1}^{n}$ and consider the space
\[
K=K_{\lambda}={\rm span}(k_{\lambda_{i}},\,i=1\dots n)=H^{2}\cap(BH^{2})^{\bot}.
\]
Observing that $B\overline{B}=1$ on the unit circle $\partial\mathbb{D}$
we get for any $g\in H^{1}$
\begin{eqnarray*}
\sum_{i=1}^{n}\overline{\alpha_{i}}w_{i} & = & \left\langle f,\,B\left(\sum_{i=1}^{n}\alpha_{i}\overline{B}k_{\lambda_{i}}+\overline{zBg}\right)\right\rangle \\
 & = & \left\langle f\left(\sum_{i=1}^{n}\overline{\alpha_{i}}B\overline{k_{\lambda_{i}}}+zBg\right),\,B\right\rangle .
\end{eqnarray*}
It is easy to check that $\sum_{i=1}^{n}\overline{\alpha_{i}}B\overline{k_{\lambda_{i}}}\in zK$
and that $zBg\in zH^{1}$. We put
\[
\varphi=\sum_{i=1}^{n}\overline{\alpha_{i}}B\overline{k_{\lambda_{i}}}+zBg=z{\varphi_{0}}
\]
where ${\varphi_{0}}\in K+H^{1}$ so that 
\[
\sum_{i=1}^{n}\overline{\alpha_{i}}w_{i}=\left\langle zf\varphi_{0},\,B\right\rangle .
\]
Following D. Sarason we transform the above right-hand side adapting his proof of \cite[Lemma 2.1]{SD} for completeness.
It is well-known \cite[Corollary 3.7.4]{NN2}  that the zeroes $(\mu_{i})_{i\geq1}$ of
$\varphi_{0}\in H^{1}$ satisfy the
Blaschke condition $\sum_{i}(1-\abs{\mu_{i}})<\infty$. Thus defining the Blaschke product $B_{0}=\prod_{i\geq1}\frac{z-\mu_{i}}{1-\overline{\mu_{i}}z}$ the function $\frac{\varphi_{0}}{B_{0}}$ belongs also to the
Hardy space $H^{1}$ and does not vanish on $\mathbb{D}$. Therefore
there exists $h\in\mathcal{H}ol(\mathbb{D})$ such that $\frac{\varphi_{0}}{B_{0}}=h^{2}.$
We write $\varphi_{0}=f_{1}f_{2}$ where $f_{1}=B_{0}h$ and $f_{2}=h$.
We obviously have $\abs{f_{1}(z)}^{2}=\abs{f_{2}(z)}^{2}=\abs{\varphi_{0}(z)}$
for any $z\in\partial\mathbb{D}$. Going back to the last expression
of $\sum_{i=1}^{n}\overline{\alpha_{i}}w_{i}$ we denote by $P$ the
orthogonal projection from $H^{2}$ to $K$ and write:
\begin{eqnarray*}
\sum_{i=1}^{n}\overline{\alpha_{i}}w_{i} & = & \left\langle zff_{1}f_{2},\,B\right\rangle \\
 & = & \left\langle zf(f_{1}-Pf_{1}+Pf_{1})f_{2},\,B\right\rangle \\
 & = & \left\langle zf(Pf_{1})f_{2},\,B\right\rangle
\end{eqnarray*}
because $f_{1}-Pf_{1}\in BH^{2}$ and $B\perp  z{f_{2}}BH^{2}.$ For the same
reason
\[
\sum_{i=1}^{n}\overline{\alpha_{i}}w_{i}=\left\langle zf(Pf_{1})(Pf_{2}),\,B\right\rangle =\left\langle zf(Pf_{1}),\,B\overline{(Pf_{2})}\right\rangle .
\]
It is easy to check that $B\overline{(Pf_{2})}\in zK$. Therefore
we put $g_{1}=Pf_{1}$, $g_{2}=\frac{B\overline{(Pf_{2})}}{z}$ :
$g_{1}$ and $g_{2}$ belong to $K$ and for $i=1,2$ we have
\[
\Norm{g_{i}}{H^{2}}\leq\Norm{Pf_{i}}{H^{2}}\leq\Norm{f_{i}}{H^{2}}=\sqrt{\Norm{\varphi_{0}}{H^{1}}}.
\]
To conclude we have
\[
\sum_{i=1}^{n}\overline{\alpha_{i}}w_{i}=\left\langle fg_{1},\,g_{2}\right\rangle ,\qquad g_{i}\in K.
\]
There exists $(\beta_{i})$ such that $g_{2}=\sum_{i=1}^{n}\beta_{i}k_{\lambda_{i}}$
and therefore
\[
\sum_{i=1}^{n}\overline{\alpha_{i}}w_{i}=\sum_{i=1}^{n}\overline{\beta_{i}}w_{i}g_{1}(\lambda_{i})=\left\langle g_{1},\,\sum_{i=1}^{n}\beta_{i}\overline{w_{i}}k_{\lambda_{i}}\right\rangle .
\]
Applying Cauchy-Schwarz inequality we find
\begin{eqnarray*}
\Abs{\sum_{i=1}^{n}\overline{\alpha_{i}}w_{i}}^{2} & \leq & \Norm{g_{1}}{H^{2}}^{2}\Norm{\sum_{i=1}^{n}\beta_{i}\overline{w_{i}}k_{\lambda_{i}}}{H^{2}}^{2}\\
 & \leq & \Norm{\varphi_{0}}{H^{1}}\frac{\Norm{\sum_{i=1}^{n}\beta_{i}\overline{w_{i}}k_{\lambda_{i}}}{H^{2}}^{2}}{\Norm{g_{2}}{H^{2}}^{2}}\Norm{g_{2}}{H^{2}}^{2}\\
 & \leq & \Norm{\varphi_{0}}{H^{1}}^{2}\frac{\Norm{\sum_{i=1}^{n}\beta_{i}\overline{w_{i}}k_{\lambda_{i}}}{H^{2}}^{2}}{\Norm{\sum_{i=1}^{n}\beta_{i}k_{\lambda_{i}}}{H^{2}}^{2}}\\
 & \leq & \Norm{\varphi_{0}}{H^{1}}^{2}\sup_{(\alpha_{i})\in\mathbb{C}^{n}}\frac{\sum_{i=1}^{n}\sum_{j=1}^{n}\frac{\overline{\alpha_{i}}\alpha_{j}\overline{w_{i}}w_{j}}{1-\overline{\lambda_{i}}\lambda_{j}}}{\sum_{i=1}^{n}\sum_{j=1}^{n}\frac{\overline{\alpha_{i}}\alpha_{j}}{1-\overline{\lambda_{i}}\lambda_{j}}}.
\end{eqnarray*}
To complete the proof it remains to recall that
\begin{eqnarray*}
\Norm{\varphi_{0}}{H^{1}} & = & \Norm{z\varphi_{0}}{H^{1}}=\Norm{\varphi}{H^{1}}\\
 & = & \Norm{\sum_{i=1}^{n}\overline{\alpha_{i}}B\overline{k_{\lambda_{i}}}+zBg}{H^{1}}=\Norm{\sum_{i=1}^{n}\alpha_{i}k_{\lambda_{i}}+\overline{zg}}{L^{1}}
\end{eqnarray*}
for any $g\in H^{1}.$
\end{proof}
\section*{Acknowledgments}
The authors are deeply grateful to the Referee and to Kenneth R. Davidson for valuable comments on an earlier version of the manuscript. The second author acknowledges the Russian Science Foundation grant 14-41-00010.

\end{document}